\documentclass[12pt]{article}

\usepackage{amsfonts,amsthm,amssymb,amsmath}
\usepackage[cp866]{inputenc}

\usepackage[T2A]{fontenc}

\usepackage{epsfig}
\usepackage{graphicx}

\newtheorem{lm}{Lemma}
\newtheorem{theorem}{Theorem}

\theoremstyle{remark}
\newtheorem{remark}{Remark}

\title{On bifurcations of symmetric elliptic orbits}

\author{
	M.\,S.\,Gonchenko$^1$	\\
	{\small
		$^1$ Departament de Matem\`atiques i Inform\`atica, Universitat  de Barcelona, Spain}\\
		}

\date{\today}

\begin{document}

\maketitle

\begin{abstract}
We study bifurcations of symmetric elliptic fixed points in the case of \emph{p}:\emph{q} resonances  with odd $q\geq 3$. We consider the case where the initial area-preserving map $\bar z =\lambda z + Q(z,z^*)$ possesses the central symmetry, i.e. is invariant under the change $z\to -z$, $z^*\to  -z^*$. We construct normal forms for such maps in the case $\lambda = e^{i 2\pi \frac{p}{q}}$, where $p$ and $q$ are mutually prime integer numbers, $p\leq q$ and $q$ is odd, and study local bifurcations of the fixed point $z=0$ in various settings. We prove the appearance of garlands consisting of four $q$-periodic orbits, two orbits are elliptic and two orbits are saddle, and describe the corresponding bifurcation diagrams for one- and two-parameter families. We also consider the case where the initial map is reversible and find conditions when non-symmetric periodic orbits of the garlands are non-conservative (compose symmetric pairs of stable and unstable orbits as well as area-contracting and area-expanding saddles).
\end{abstract}

\section{Introduction}

In the present paper we study local bifurcations of area-preserving maps which can be written in the complex coordinates $z=x+iy$ and $z^*=x-iy$  ($x,y \in \mathbb{R}$) as
\begin{equation}
\bar z = f(z, z^*) = \lambda \left(z + \sum\limits_{m+k\geq 2,\; m,k\geq 0} A_{m,k} z^m (z^*)^k\right),
\label{eq:2map}
\end{equation}
where eigenvalue $\lambda$ and coefficients $A_{m,k}$ are complex. Our main assumption are as follows
\begin{itemize}
\item[($A_1$)]
map (\ref{eq:2map}) is invariant under the change  $z\to -z$ (automatically $z^*\to -z^*$), i.e. $f(z,z^*)=-f(-z, -z^*)$. We say that the map possesses the so-called \emph{central symmetry}. This means that in the real plane $(x,y) \in \mathbb{R}^2$, the phase portraits are symmetric with respect to the change $x\to -x, y\to -y$;
\item[($A_2$)]
the eigenvalue takes the form $\displaystyle \lambda = e^{ i 2\pi \frac{p}{q}}$, where $p$ and $q$ are mutually prime natural numbers which satisfy $p<q$;
\item[($A_3$)] the denominator $q$ is an odd natural number such that $q\geq 3$.
\end{itemize}

We note that assumptions $(A_1)-(A_3)$ allow us to significantly simplify the normal form of map~(\ref{eq:2map}). Namely, we show that when $q\geq 3 $ is odd, the map can be brought by polynomial changes of coordinates to the form (see Section~\ref{sec:NFRes})
\begin{equation}
\displaystyle \bar z = e^{2\pi\frac{p}{q}i} \left( z + i\Omega(|z|^2)z + A_{0,2q-1}\cdot (z^*)^{2q-1}\right) + O\left(|z|^{2q+1}\right),
\label{eq:2mapnf0}
\end{equation}
where $\Omega(|z|^2) = g_1 |z|^2(1 + O(|z|^2)$ is a real-valued function. We assume also that
\begin{itemize}
\item[($A_4$)] the coefficients of the normal form~\eqref{eq:2mapnf0} satisfy $g_1 \neq 0 , \;\;  A_{0,2q-1} \neq 0.$
\end{itemize}

Under assumptions $(A_1)-(A_4)$, the fixed point $z=0$ of map~\eqref{eq:2map} is elliptic, and we call this point \emph{ central symmetric elliptic point of odd order}. 
In the present paper we study bifurcations of this fixed point in various settings. 

For two-dimensional symplectic maps, the basis of phenomenon of the $p$:$q$ resonance consists in a local bifurcation of a fixed (periodic) point with eigenvalues
$\lambda_{1,2} = e^{\pm i\;2\pi \frac{p}{q}}$, where $p$ and $q$ are mutually prime natural numbers and $p<q$. Among them, the so-called strong resonances, i.e. the bifurcation phenomena connected with the existence of periodic points with eigenvalues  $e^{\pm i\;2\pi/q}$ with $q=1,2,3,4$ (respectively, the 1:1, 1:2, 1:3 and 1:4 resonances), play a special role in the dynamics of area-preserving maps \cite{Arn73, AKN06} and demonstrate very specific bifurcations~\cite{Bir87, MS94, DM00, Gon05, GGO17,GGOV18}. On the other hand, the non-degenerate resonances with $q\geq 5$ are all similar in the sense that in a general one-parameter family of area-preserving maps, bifurcations of $p$:$q$ resonant fixed points lead to the appearance of a garland (stability islands) consisting of two $q$-periodic orbits of saddle and elliptic type, see Fig.~\ref{bif1p2q}.
In this case, one considers a real parameter $\mu$ which varies the angle of eigenvalues of fixed point such that the perturbed
eigenvalues are $\lambda_{1,2}(\mu) = e^{\pm i\;2\pi (\frac{p}{q}+\mu)}$.  Then, as well-known \cite{Arn73,Tak74,LQ94}, the study of local bifurcations of the initial map can be reduced to the study of bifurcations of equilibria in the corresponding flow of the form
\begin{equation}\label{eq40}
\dot z =
i\mu z + i \Phi(|z|^2)z + \delta (z^*)^{q-1} + O\left(|z|^{q+1}\right),
\end{equation}
where $\Phi(|z|^2)$ is a real-valued function such that $\Phi(0)=0$. The flow~\eqref{eq40} is called the \emph{standard flow normal form} in the case of $p$:$q$ resonance with $q\geq 3$. When the genericity conditions  $\Phi^\prime(0) = l_1 \neq 0$ and $\delta\neq 0$ are fulfilled in the case of resonance with $q\geq 5$, there emerges a conservative garland consisting of $q$ saddles and $q$ centers, as shown in Fig.~\ref{bif1p2q}. Accordingly, for the map, this means the appearance of a garland which consists of $q$-periodic  saddle and elliptic orbits. In the continuous case, we denote this kind of garland as $G_{q,q}$ where the subscript indexes stand for the numbers of the corresponding saddles and centers. In the discrete setting, the corresponding garland is $G^{1,1}_{q,q}$, where the superscript indexes indicate the numbers of the saddle and elliptic orbits that compose the garland and the subscript indexes stand for the periods of these orbits.

Note that the case with $q=3$ is completely different. It corresponds to the strong resonance of order 3.  Main bifurcations of such resonance in the conservative setting was first studied by Arnold~\cite{Arn73}, see also~\cite{AKN06, SV09, BHJVW03,GKSS22}. When $\delta\neq 0$, the zero equilibrium at $\mu =0$ is a saddle point with 6 separatrices that at varying $\mu$ becomes a center surrounded by 3 saddles, see Fig.~\ref{bif1p2q}, bottom plots, where reconstructions of phase portraits are shown when passing through nondegenerate 1:3 resonance 

\begin{figure}[t]
	\centerline{\includegraphics[width=10cm]{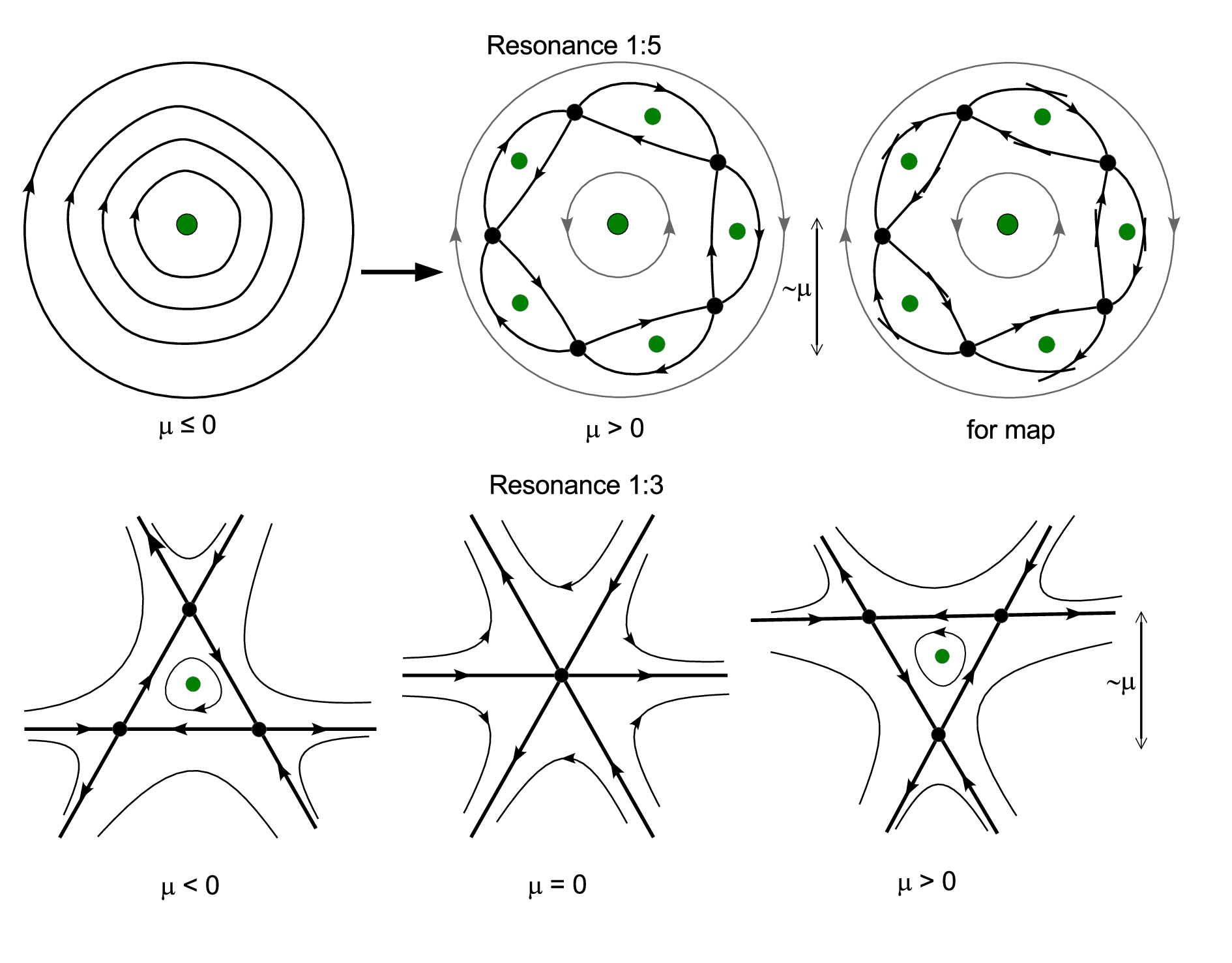}}
	\caption{{\footnotesize {A generic conservative bifurcation of the $p$:$q$ resonance with $q\geq 3$. As an example, we represent the cases of 1:5 resonance, top plots, and 1:3 resonance, bottom plots. The plots are related to flow~(\ref{eq40}), except for the top right plot where we illustrate the results for the initial map. 
	For the flow, the 1:5 resonance is as follows:
	for $l_1\mu>0$ only the equilibrium $z=0$ exists (top left plot); for $l_1\mu <0$ a garland $G_{q,q}$ around $z=0$ appears containing $2q$ equilibria: $q$ saddles and $q$ centers (top center plot); for a map $G_{q,q}$ corresponds to a garland $G_{q,q}^{1,1}$ which contains two saddle and elliptic $q$-periodic orbits (note that unlike the flow case the invariant manifolds of saddles split in general).
	A nondegenerate 1:3 resonance has the following structure: for $\mu=0$ point $z=0$ is a saddle with 6 separatrices and for $\mu\neq 0$ it becomes a center surrounded by 3 saddle equilibria connected by a heteroclinic contour of different orientation for $\mu<0$ and $\mu>0$.}}}
\label{bif1p2q}
\end{figure}

It is also worth mentioning that for even $q$ the truncated normal form of map~(\ref{eq40}), i.e. without the $O$-terms, is automatically centrally symmetric and, therefore, the assumption ($A_1$) on the central symmetry for map~(\ref{eq:2map}) is not necessary.
However, when $q$ is odd, this assumption implies that $\delta =0$ which means that the centrally symmetric $p$:$q$ resonance for odd $q$ is degenerate.

In~\cite{GLRT14,GT17}, bifurcations of degenerate $p$:$q$ resonances (with $\delta=0$) were studied in the reversible case, i.e. when the map $f$ and the inverse map $f^{-1}$ are conjugate by means of an involution $h$, i.e. $f=h\circ f^{-1} \circ h$ and $h^2=Id$,  and  perturbations preserve the reversibility but destroy the conservativity. Note that the reversibility implies symmetry of orbits. We call an orbit \emph{reversibly symmetric} if it intersects the set of fixed points of the involution $Fix(h)=\{x\in \mathbb{R}^2: h(x)=x\}$ or the set $Fix(h\circ f)$. A reversibly symmetric orbit is saddle if the eigenvalues are $0<\lambda<1$ and $\lambda^{-1}$ and is elliptic if the eigenvalues are $\lambda_{1,2}=e^{\pm i \phi}, \phi\neq 0, \pi$.  Reversibly non-symmetric orbits can be of any type and they appear in pairs: for a reversible non-symmetric orbit $P_1$ with eigenvalues $\lambda_1$ and $\lambda_2$ there exists a symmetric orbit $P_2$ (i.e. $h(P_1)=P_2$ and $h(P_2)=P_1$) with eigenvalues $\lambda_1^{-1}$ and $\lambda_2^{-1}$. We say that they compose \emph{reversibly symmetric pairs}. Usually in reversible systems non-symmetric orbits emerge via symmetry-breaking bifurcations.

In the reversible non-conservative case, the flow normal form for the $p$:$q$ resonance can be presented as follows~\cite{GLRT14,GT17}
\begin{equation}\label{eq40a}
\dot z =
i\mu_1 z + i \Phi(|z|^2)z + i\mu_2 (z^*)^{q-1} + iA z^{q+1} + iB z (z^*)^q + O\left(|z|^{q+2}\right),
\end{equation}
where $\mu_1$ and $\mu_2$ are real parameters, $A$ and $B$ are real coefficients, 
\begin{equation}\label{eq:Phi} 
\Phi(z^2) = l_1 |z|^2+ l_2 |z|^4 + \ldots + l_n |z|^{2n} + \ldots 
\end{equation}
is a real-valued function with $l_1\neq 0$. In this case, flow~(\ref{eq40a}) is reversible, i.e. invariant under time reversal and the involution $z\to z^*$,
but not conservative if $B\neq A(q+1)$.
In \cite{GLRT14,GT17} it was proven that, depending on a relation between $A$ and $B$, 
bifurcations in~(\ref{eq40a}) lead to the appearance of a garland which contains either (i) $2q$ reversibly symmetric conservative saddle equilibria, $q$ asymptotically stable (sinks) and $q$ completely unstable (sources) equilibria or (ii) $2q$ reversibly symmetric conservative centers, $q$ saddles with negative divergence 
and  $q$ saddles with positive divergence.  We denote this non-conservative garland as $G'_{2q,q,q}$ with the indexes standing for the numbers of reversibly symmetric saddles/centers and nonconservative  equilibria. 
Accordingly, in the discrete case, this means the emergence of an equivalent garland $G'_{2q,q,q}$ with four $q$-periodic orbits, the two orbits are reversibly symmetric and conservative saddle or elliptic orbits, and the other two orbits are non-conservative and compose reversibly symmetric pairs of stable and unstable orbits or area-contracting (with the Jacobian $J< 1$) and area-expanding (with the Jacobian $J>1$) saddles.   
This result can be viewed as one of principally important results in the theory of mixed dynamics \cite{G16,GT17,GGK20,G21}.

Recall that the mixed dynamics is a new third form of dynamical chaos additional to the two well-known forms, conservative and dissipative chaos.
The fundamental property of mixed dynamics is that attractors and repellers have a non-empty intersection. Moreover, they intersect along closed invariant sets, the so-called \emph{reversible cores}, that attract nothing and repel nothing. In this case, any orbit near the core never leaves its neighbourhood: it only tends to the nearest attractor
as $t\to +\infty$ or to the nearest repeller as $t\to -\infty$. It was established also in~\cite{GLRT14,GT17} that, for generic two-dimensional reversible diffeomorphisms, any reversibly symmetric elliptic periodic orbit is a reversible core and the limit of periodic stable and unstable orbits. Note that one can observe mixed dynamics in applications~\cite{PT, Kaz13, GGKT17, Kuz17, EN19, GGK20, AS20}. The existence of Newhouse domains with mixed dynamics was also proved in~\cite{GTS97,LS04,DGGLS13,Tur15,DGGL18,GGS20}.  

The paper is organized as follows. In Section~\ref{sec:mainresults}, we expose and announce our main results.  In Section~\ref{sec:NFRes}, we derive the $2q$-order local normal forms and the corresponding flow normal forms near the $p$:$q$ resonant elliptic fixed point, when $q$ is odd. The rest of the paper is dedicated to the study of bifurcations in the flow normal forms in different settings. Namely, in Section~\ref{sec:SymmCase}, we consider the centrally symmetric case and establish the emergence of a symmetric garland. In Section~\ref{sec:pfork}, we study conservative symmetry-breaking bifurcations and prove the appearance of a conservative garland with non-symmetric conservative orbits. Finally, in Section~\ref{sec:NonCons}, we discuss reversible symmetry-breaking bifurcations which could lead to the birth of a non-conservative garland that contains non-symmetric orbits: either stable and unstable orbits or non-conservative saddles.

\section{Main results}\label{sec:mainresults}

In the present paper we study bifurcations of degenerate $p$:$q$ resonant elliptic fixed point with odd $q\geq 3$, but in a setting different from~\cite{GLRT14,GT17}.
We deal with a two-dimensional area-preserving map of the complex form~\eqref{eq:2map} under assumptions $(A_1) - (A_4)$ and study bifurcations of the fixed point under various types of perturbations which (a) preserve the central symmetry, (b) destroy the central symmetry but preserve the reversibility and conservativity, and (c) destroy the symmetry and conservativity and keep the reversibility. In all the cases, we construct the bifurcation diagrams in the corresponding one- and two-parameter families and pay attention to the appearance of garlands

We study bifurcations in one-parameter families of reversible area-preserving  maps  $\bar z = f(z,z^*,\mu)$ which preserve the symmetry $f(z,z^*,\mu) = -f(-z,-z^*,\mu)$. We show that in this case the corresponding flow normal form can be written as
\begin{equation}\label{eq401}
\dot z =
i\mu z + i \Phi(|z|^2)z + i \alpha (z^*)^{2q-1} + O\left(|z|^{2q+1}\right),
\end{equation}
where $\mu$ is a real parameter, $\alpha$ is a real coefficient and $\Phi$ is a real-valued function of the form~\eqref{eq:Phi}. 
Assumption ($A_4$) implies that 
\begin{equation}
l_1\neq 0, \;\;\; \alpha\neq 0.
\label{eq:nondeg_l1delta}
\end{equation}
We prove the following result.

\begin{figure}[t]
	\centerline{\includegraphics[width=8cm]{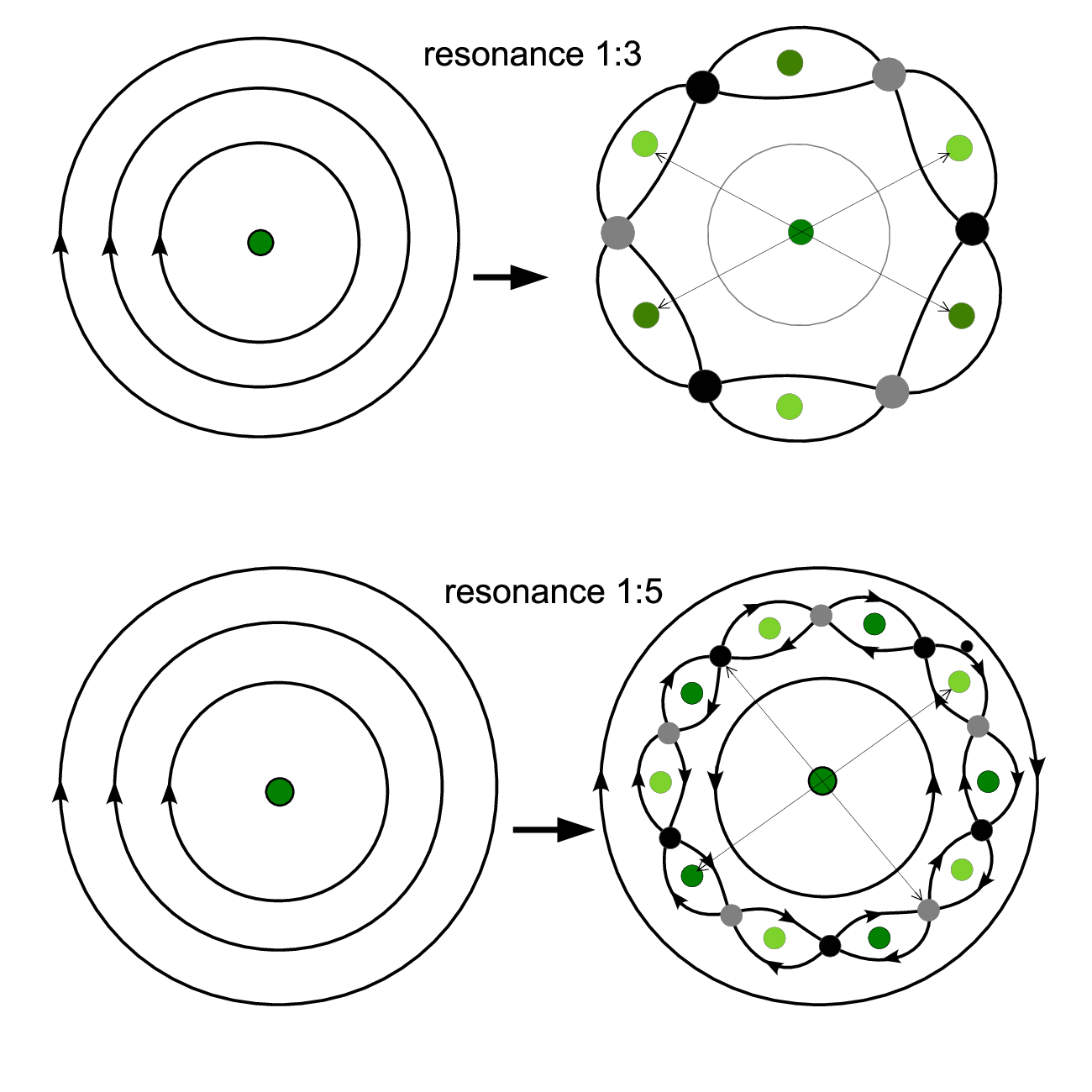}}
	\caption{{\footnotesize {Conservative bifurcations of the centrally symmetric 1:3 and 1:5 resonances. In both cases, before the bifurcation, there is only the elliptic resonant fixed point (with multipliers $e^{\pm i 2\pi/3}$ and $e^{\pm i 2\pi/5}$, respectively). After the bifurcation, there appear garlands $G_{6,6}$ and $G_{10,10}$, respectively. For the map, these garlands correspond to $G_{3,3}^{2,2}$ and $G_{5,5}^{2,2}$ which contain two saddle and two elliptic orbits of period $3$ and $5$, respectively.}
			{One (black) saddle  orbit  is centrally symmetric to the other (grey) saddle as well as the elliptic (dark green and light green) orbits compose a pair of mutually centrally symmetric orbits. The resulting garland is centrally symmetric.}
			{One can see that centrally symmetric bifurcations in the cases of the 1:3 resonance and $p$:$q$ resonances with odd $q\geq 5$ occur in the same way, compare with Fig.~\ref{bif1p2q}.}
	}}
	\label{bifpdf1p4q}
\end{figure}

\begin{theorem}
Let $U$ be a sufficiently small fixed neighbourhood of $z=0$. Then,
for $\mu$ sufficiently small,   flow~(\ref{eq401}), with function $\Phi$ defined in~\eqref{eq:Phi} and odd $q\geq 3$, has the following dynamics in $U$:
\begin{itemize}
\item if $\mu l_1 > 0$, there is only one center equilibrium; \item if at $\mu l_1 <0 $, there appears a garland $G_{2q, 2q}$ which consists of $4q$ equilibria, alternating $2q$ saddles and $2q$ centers, which lie nearly the circle $|z| = - \mu/l_1$ and are equidistant from each other at the angle $\pi/2q$.
\end{itemize}
\label{th:th1}
\end{theorem}

After the bifurcation, the flow~\eqref{eq401} possesses a garland $G_{2q,2q}$ consisting of $4q$ alternating saddle and elliptic equilibria uniformly located around the point $z=0$, see Fig.~\ref{bifpdf1p4q}. Accordingly, for the map, this means the emergence of a garland $G^{2,2}_{q,q}$ with four orbits of period $q$, two of them are saddles and the other two are elliptic orbits. One of these orbits is symmetric to the other orbit of the same type with respect to the origin (after the change $x\to -x$ and $y\to -y$). We call them \emph{mutually centrally symmetric} orbits. Namely,  these periodic orbits are not centrally symmetric, since they are of odd periods, but along with the other orbits (of the same type) they compose two  pairs of mutually centrally symmetric orbits.

After that, we study bifurcations of symmetric elliptic points in families $\bar z= f(z,z^*, \mu_1, \mu_2)$ that destroy the central symmetry, i.e. $f(z,z^*, \mu_1, \mu_2) \neq -f(z,z^*, \mu_1, \mu_2)$ but keep conservativity and reversibility.
In this case, it is natural to consider the following perturbation of  system~(\ref{eq401}), compare with~(\ref{eq40}),
\begin{equation}\label{eq402}
\dot z =
i\mu_1 z + i \Phi(|z|^2)z + i\mu_2 (z^*)^{q-1}  + i\alpha (z^*)^{2q-1} + O\left(|z|^{2q+1}\right) + O\left(|\mu||z|^{q+1}\right)
\end{equation}
where $\mu_1$ and $\mu_2$ are real parameters, $\alpha\neq 0$ is real. When $q$ is odd and $\mu_2\neq 0$, the system loses the central symmetry, 
however,
the normal form remains conservative  and reversible. 
We obtain the following result for this family.

\begin{theorem}
Let $V_0$ be a small neighbourhood of the origin in the $(\mu_1,\mu_2)$-parameter plane. In $V_0$ there exist bifurcation curves
\begin{equation}\label{eq540}
\begin{array}{l}
L_{pf}=\left\{(\mu_1, \mu_2)\in V_0 :\; \mu_2 = \pm 2\alpha \left(-\frac{\mu_1}{l_1}\right)^{q/2}(1+O(\mu_1)) \right\}, \\ L_{p:q}=\{(\mu_1, \mu_2)\in V_0: \mu_1=0\}
\end{array}
\end{equation}
which correspond  to a pitchfork bifurcation of $q$ centrally symmetric nonzero equilibria and  the $p$:$q$~resonance, respectively. Curves $L_{pf}$  and
$L_{p:q}$ divide $V_0$ into 4 domains $I-IV$ where~(\ref{eq402}) exhibits the following dynamics  
\begin{itemize}
\item[(a)] For  $(\mu_1,\mu_2)\in I$,  there are no nonzero equilibria; 
\item[(b)] For  $(\mu_1,\mu_2)\in II\cup IV$, there is  a garland $G_{q,q}$ that contains $q$ centrally symmetric saddles and $q$ centrally symmetric centers; 
\item[(c)] For $(\mu_1,\mu_2)\in III$, there are $4q$ nonzero  equilibria which form a garland $G_{2q,2q}$.
\begin{itemize}
\item If $l_1\alpha >0$, at crossing the curve $L_{pf}$, $q$ centrally symmetric centers undergo a supercritical pitchfork bifurcation. As a result there appear $2q$ centrally symmetric saddles and $2q$ non-symmetric centers inside the garland $G_{2q,2q}$. In domain $III$ there are two bifurcation curves $g_1$ and $g_2$ related to heteroclinic connections between the $2q$ saddles, and, thus, in the subdomain of $III$, between the curves $g_1$ and $g_2$ 
garland $G_{2q,2q}$ rotates.

\item If $l_1\alpha <0$, at crossing the curve $L_{pf}$, $q$ centrally symmetric saddles undergo a subcritical pitchfork bifurcation after which in domain $III$, there appears a garland $G_{2q,2q}$ that contains $2q$ centrally symmetric centers and $2q$ non-symmetric saddles.
%
\end{itemize}
\end{itemize}\label{th:th2}
\end{theorem}

\begin{figure}[t]
\centerline{\includegraphics[width=16cm]{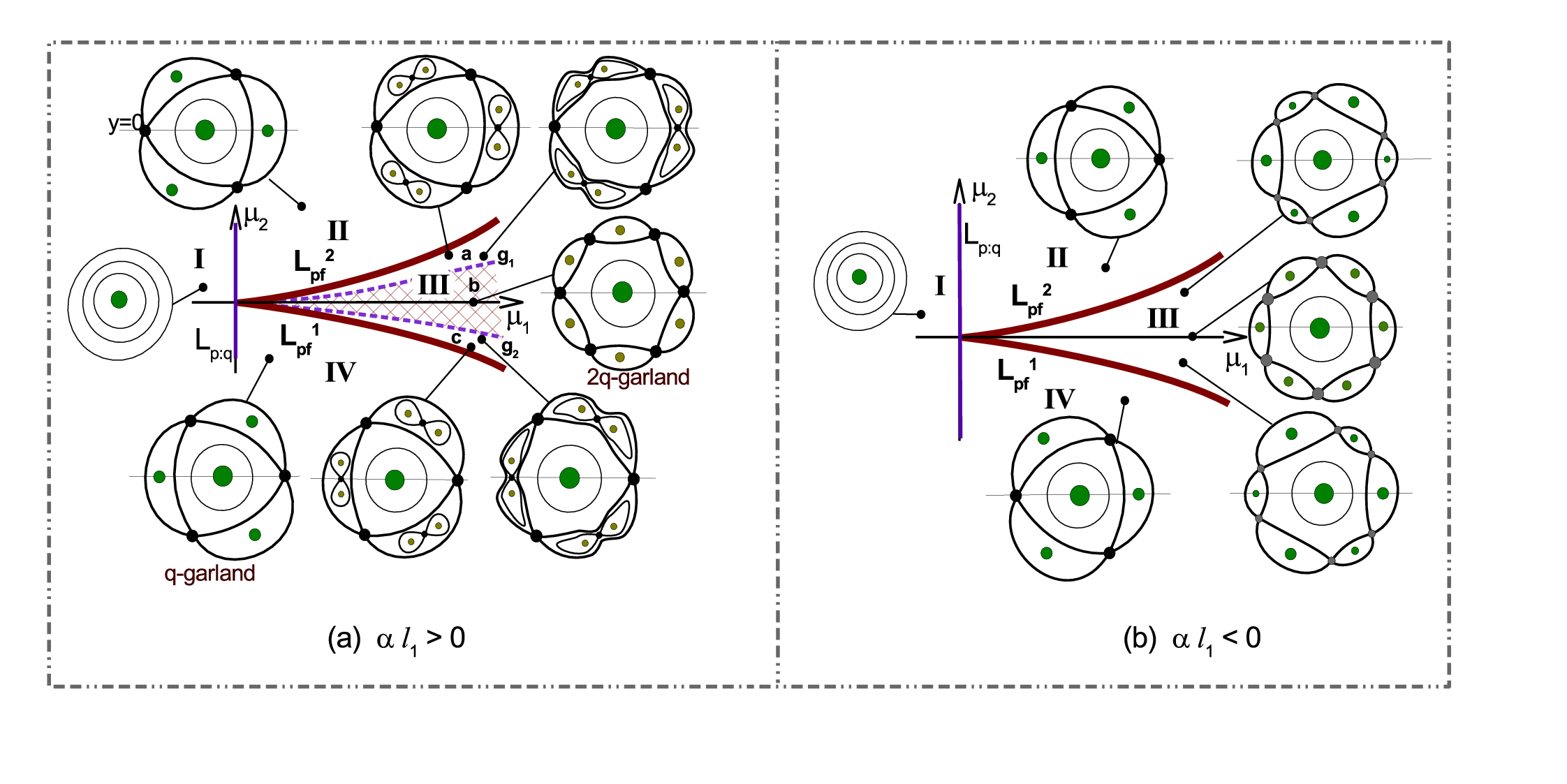}}
\caption{{\footnotesize A conservative bifurcation of the $p$:$q$ resonance with odd $q\geq 3$ for a two-parameter family where the central symmetry is broken for $\mu_2=0$. Two bifurcation diagrams for the 1:3 resonance are illustrated for the cases: (a) $l_1\alpha >0$ and (b) $l_1\alpha <0$. Note that Theorem~\ref{th:th1} corresponds to the case $\mu_2=0$
}}
\label{bdiag1p3n2}
\end{figure}

One can see a certain difference between the cases $l_1\alpha >0$ in Fig.~\ref{bdiag1p3n2}a and $l_1\alpha <0$ in Fig.~\ref{bdiag1p3n2}b. In the first case, $l_1\alpha >0$, the transition from garland $G_{q,q}$ to garland $G_{2q,2q}$ takes place after the supercritical pitchfork bifurcation in curve $L_{pf}$, however, later heteroclinic bifurcations (curves $g_1$ and $g_2$) take place when the separatrices of $q$ interior saddles connect the separatrices of $q$ exterior saddles, and the phenomenon of splitting of separatrices occurs (see~\cite{DGG16,DGG20} and the references therein for more details about this phenomenon). In the second case, $l_1\alpha >0$, the transition from  $G_{q,q}$ to $G_{2q,2q}$ holds after the subcritical pitchfork bifurcation in the curve $L_{pf}$  when there appear $q$ new islands with 2 saddles and 1 center.

Accordingly, for the map, Theorem~\ref{th:th2} indicates the bifurcation mechanism of transition from the standard symmetric garland $G_{q,q}^{1,1}$ to a garland $G_{q,q}^{2,2}$ which additionally contain two centrally non-symmetric orbits of period $q$.

The results of Theorem~\ref{th:th2} can be further generalized to the case of perturbations that destroy both the symmetry and conservativity but keep reversibility. In Section~\ref{sec:NonCons}, we discuss symmetry-breaking bifurcations for reversible perturbations which lead to the appearance of garlands with reversibly non-symmetric orbits. We show that reversibly non-symmetric equilibria appear due to reversible pitchfork bifurcations~\cite{LT12}. As a result of this bifurcation, there can emerge two different kinds of garlands: $G'_{2q,q,q}$ which consists of $2q$ conservative reversibly symmetric centers and $q$ reversibly symmetric pairs of saddles with positive and negative divergence (similar to the one in Fig.~\ref{bdiag1p3n2}b, but the saddles are not conservative) o  a new type of garland $G'_{2q,q,q}$ which consists of $2q$ conservative reversibly symmetric saddles and $q$ reversibly symmetric pairs of stable and unstable foci (see Section~\ref{sec:NonCons} and the bifurcation diagram in Fig.~\ref{bdrev1p3v2}).
For the map, these garlands correspond to discrete garlands $G'_{2q,q,q}$ with two reversibly symmetric $q$-periodic elliptic (saddle) orbits and $q$-periodic non-conservative saddles (stable and unstable foci, respectively) which compose reversibly symmetric pairs.

\section{Normal forms for the centrally symmetric $p$:$q$ resonance}\label{sec:NFRes}

We consider a two-dimensional
map of the form~(\ref{eq:2map}). We assume that the map satisfies $(A_1)-(A_4)$.
When studying dynamics and bifurcations of the fixed point $z=0$ (when $|\lambda|=1$), it is important to bring map~(\ref{eq:2map}) to the form as simple as possible by means of sufficiently smooth changes of coordinates, ideally, to the so-called normal form which contains only resonant terms.

We say that in~(\ref{eq:2map}) the monomial $A_{m,k} z^m (z^*)^k$ is resonant if it cannot vanish by any $C^r$-change of coordinates with $r> m+k$ in any neighborhood of $z=0$. Let us consider  the following auxiliary map
$$
\bar z = \lambda z + A_{m,k} z^m (z^*)^k.
$$
It is well-known that the nonlinear term $A_{m,k} z^m (z^*)^k$ is resonant if the following resonant condition is true
\begin{equation}
\lambda^m (\lambda^*)^k -\lambda = 0.
\label{eq:rezcond}
\end{equation}
This can be seen by means of change $z = w + C_{m,k} w^m (w^*)^k$. After this change, the map takes the form
$$
\bar w = \lambda w +  \left[ \lambda C_{m,k} - C_{m,k}\lambda^m (\lambda^*)^k +  A_{m,k}\right] w^m (w^*)^k + O(|w|^{k+m+1}).
$$
If~\eqref{eq:rezcond} does not hold, one can take 
$$
C_{m,k} = \frac{A_{m,k}}{\lambda^m (\lambda^*)^k -\lambda}
$$
in order to nullify the square bracket, and, thus, the corresponding term is no resonant. The square bracket does not vanish if the resonant condition~\eqref{eq:rezcond} is fulfilled, and, therefore, the monomial $A_{m,k} z^m (z^*)^k$ is resonant.

In the present paper we consider the case $
\displaystyle \lambda = e^{i\;2\pi\frac{p}{q}},
$
where $p$ and $q$ are mutually prime natural integers, $p<q$ and $q\geq 3$. Then, condition~(\ref{eq:rezcond}) implies that
$$
e^{i 2\pi \frac{p}{q} (m -1 - k)} = 1,
$$
and, since $\frac{p}{q}$ is an irreducible fraction, we obtain 
\begin{equation}
m-1-k = jq
\label{eq:rez0}
\end{equation}
for some integer $j$.

For map~(\ref{eq:2map}) that satisfies $(A_1)-(A_4)$, 
condition~\eqref{eq:rez0} immediately implies that if $m+k$ is even, term $A_{m,k} z^m (z^*)^k$ is not resonant and can vanish after some change of coordinates. Hence, in a map of the form~(\ref{eq:2map}) with $(A_1)-(A_4)$, we can nullify all monomials of even degrees. Thus, we can additionally assume
\begin{equation}
A_{m,k} \equiv 0, \;\; \mbox{when $m+k$ is even}.
\label{eq:zeros}
\end{equation}
This property greatly simplifies the problem of studying bifurcations of $p$:$q$ resonant elliptic points,  when $\displaystyle \lambda = e^{i\;2\pi \frac{p}{q}}$, $p$ and $q$ are mutually prime natural numbers, $p<q$ and $q$ is {\em odd}.

Let us play further with the resonant condition (\ref{eq:rez0}) in order to obtain all the resonant terms up to order $2q$ taking into account the identities~(\ref{eq:zeros}). We can see the following:

1) For $j=0$, one has $k = m-1$  and, thus, the terms
$$
A_{m,m-1}z^m(z^*)^{m-1} = A_{m,m-1}|z|^{2m-2} z, m=2,3,...,
$$
are nonremovable by any change and resonant. These are the so-called {\em identical resonances}. Note also that all such monomials have odd degrees.

2) For $j=1$ in~(\ref{eq:rez0}), one gets $m = q  + k+ 1$. Then for all possible $k=0,1,...,n,...$, we obtain non-resonant monomials $A_{q+k+1,k} z^{q+k+1}(z^*)^k$
with degrees $m+k=q+2k+1$ which are always even if $q$ is odd. Thus, by~(\ref{eq:zeros}), all these monomials have zero coefficients:  $A_{q+k+1,k}\equiv 0$.

3) For $j=-1$ in~(\ref{eq:rez0}), $k = q + m -1$. For all possible $m=0,1,...,n,...$, we get the non-resonant monomials $A_{m,m+q-1} z^m (z^*)^{m+q-1}$
with even degrees $m+k=2m +q-1$ if $q$ is odd. Thus,
condition~(\ref{eq:zeros}) again implies that the coefficients of all these monomials are zero: $A_{m,m+q-1}\equiv 0$.

4) For $j=2$ in~(\ref{eq:rez0}), one gets 
$m = 2q +1 +k$. For $k=0$, we obtain the resonant monomial
$A_{2q+1,0} z^{2q+1}$
of an odd degree. Note that for $k\geq 1$ and $j=2$, the other resonant monomials have degrees greater than $2q+3$.

5) For $j=-2$ in~(\ref{eq:rez0}), one has $k = 2q + m -1$. For $m=0$, 
we get the resonant monomial $A_{0,2q-1} (z^*)^{2q-1}$. 
The other resonant monomials for  $m\geq 1$ have degrees greater $2q$.

For the other $|j|\geq 3$, it is easy to show that the resonant terms have degrees more than $3q-1$ and, therefore, they are not included in the $2q$-order normal form. 
Hence, we note that among the non-identical resonances, the resonant term $A_{0,2q-1} (z^*)^{2q-1}$ is of  the smallest degree. By assumption~$(A_4)$, we know that $A_{0,2q-1} \neq 0$.

Therefore, the main $2q$-order normal form of map (\ref{eq:2map}) with $(A_1)-(A_4)$ can be written in the form~\eqref{eq:2mapnf0},  
where $i \lambda \Omega(|z|^2) z$ includes all the identical resonances up to order $2q$. Note that $\Omega(t)$ is a real-valued function such that $\Omega(0) = 0$.

As well-known, the system of ODEs
\begin{equation}\label{eq4}
\dot z =
\tilde\Phi(|z|^2)z + \delta (z^*)^{2q-1} + O\left(|z|^{2q+1}\right) ,
\end{equation}
where $\tilde\Phi = e^{- i\; 2\pi p/q} \Omega$ and $\delta = e^{- i\;2\pi p/q} A_{0,2q-1}$, can be considered as the flow normal form for map~(\ref{eq:2mapnf0}) near the fixed point $z=0$.
Indeed, if we consider the map $\tilde f = f\circ R_\phi$, where
$R_\phi$ is the rotation by the angle $\phi = - 2\pi p/q$, we obtain that $\tilde f$ takes the form $\bar z =  z + \Phi(|z|^2)z + \delta (z^*)^{2q-1} + O\left(|z|^{2q+1}\right)$ (we need to make change $z \to z e^{- i\;2\pi p/q}$ in the right hand side of~(\ref{eq:2mapnf0})). Thus, $\tilde f$ is embedded into flow~(\ref{eq40}), and the dynamics of map $f$ is easily recovered if we know the dynamics of $\tilde f$. In turn, the latter can be studied via the analysis of bifurcations of flows close to flow~(\ref{eq4}), see, for example, the flow normal forms~\eqref{eq401}, \eqref{eq402} and~\eqref{eq403}.

We recall that the system of the form $\dot z = F(z,z^*), \dot z^* = F^*(z,z^*)$ is Hamiltonian if the divergence is identical zero. Namely,
$$
\mbox{Div} \equiv \frac{\partial \dot z}{\partial z} + \frac{\partial \dot z^*}{\partial z^*} \equiv 0.
$$
For flow~(\ref{eq4}), the complex conjugation has the form
$\dot z^* = \Phi^*(|z|^2)z^* + \delta^* (z)^{2q-1} + O\left(|z|^{2q+1}\right)$. Then we find that system~(\ref{eq4}) is Hamiltonian if
$\tilde\Phi + \tilde\Phi^* \equiv 0$ and, therefore, $\tilde\Phi$ should be purely imaginary  and contain only the terms corresponding to the identical resonances of the normal form~\eqref{eq:2mapnf0}. For example, we can write 
$
\tilde\Phi(|z|^2) = i \Phi(|z|^2),
$
where $\Phi$ is given in~\eqref{eq:Phi} with all the coefficients $l_i$ being real. By assumption~$(A_4)$, the non-degeneracy  
conditions~\eqref{eq:nondeg_l1delta}  
are fulfilled.

We would like to give 
some remarks on embedding of map~(\ref{eq:2map}) into flow.

\begin{remark}
Let us consider map $f$ of the form~\eqref{eq:2map} with~$(A_1)-(A_4)$ 
It is well known that map $f^q$ can be locally embedded into flow. However, certain difficulties take place with calculation of the nonlinear terms of $f^q$.
Instead of this, one can embed into flow the auxiliary map $\tilde f = f\circ R_\phi$, where $R_\phi$ is the linear rotation by angle $\phi = - 2\pi\frac{p}{q}$. Evidently, if we
consider map $\tilde f$, we do not lose information on dynamics of map $f$, and, moreover, we can extract this information from the corresponding flow. Besides, map $\tilde f$ is easily written in the form
$$\bar z =  z +  \sum \alpha_{m,k} z^m (z^*)^k,$$
where
$$
\alpha_{m,k} = A_{m,k}  e^{i\; 2\pi\frac{p}{q}(k-m)}. 
$$
\end{remark}

\begin{remark}[embedding into flow in the resonant case]
Consider map $f$ of the form
\begin{equation}
\bar z = \nu z +  \sum A_{m,k} z^m (z^*)^k,
\label{eq:2mapnf101}
\end{equation}
where $\nu^q = 1$ for the integer $q\geq 3$, i.e. we can take $\nu = e^{i 2\pi\frac{p}{q}}$. We also assume that~\eqref{eq:2mapnf101} contains only \emph{resonant monomials} $A_{m,k} z^m (z^*)^k$ which means, by virtue of~(\ref{eq:rez0}), that $m-k = jq + 1$.

First, let us consider map $\tilde f = f\circ R_{\phi}$, where $\phi = -2\pi\frac{p}{q}$. Then map $\tilde f$ takes the following form
after the changes $z\to \nu^{-1} z$ and $z^*\to \nu z^*$ in the right hand side of~(\ref{eq:2mapnf101}) 
\begin{equation}
\bar z = z +  \sum A_{m,k} z^m (z^*)^k \nu^{k-m} =  z +  \nu^{-1}\sum A_{m,k} z^m (z^*)^k ,
\label{eq:2mapnf201}
\end{equation}
since $m-k = jq + 1$ and, thus, $\nu^{k-m} = \nu^{-jq-1}= \nu^{-1}$. Evidently, map (\ref{eq:2mapnf201}) is embedded near $z=0$ into the flow
\begin{equation}
\dot z = \nu^{-1}\sum A_{m,k} z^m (z^*)^k ,
\label{eq:2mapnf20fl}
\end{equation}

Second, one can show that 
map $f^q$ can be embedded into flow of the form
$$\dot z = q \nu^{-1}\sum A_{m,k} z^m (z^*)^k ,
$$
which coincides with flow~(\ref{eq:2mapnf20fl}) after the time scaling $t_{new} = q t$.  \\
\end{remark}

\section{Bifurcations in the centrally symmetric case. Proof of Theorem~\ref{th:th1}.}\label{sec:SymmCase}

In the class of conservative systems with the central symmetry $z\to -z, z^*\to -z^*$, conditions~(\ref{eq:nondeg_l1delta}) are generic and, thus, bifurcations of a symmetric fixed point with eigenvalues~$e^{\pm i\varphi}$, where $\varphi = 2\pi p/q$ and $q$ is odd, can be studied in one-parameter families which unfold the degeneracy corresponding to $\varphi$ being rational multiple of $2\pi$, i.e. $\varphi =  2\pi p/q$. This means that we can consider the one-parameter family~\eqref{eq401},
where $\mu$ is a real parameter related to a perturbation of $\varphi$, $\Phi$ is a real-valued function corresponding to the identical resonances of~\eqref{eq:2map}, and $i\alpha$ is the coefficient of the least-order non-identical resonance. We  also write the complex conjugate 
$$\dot z^* =
-i\mu z^* - i\Phi(|z|^2)z^* - i \alpha z^{2q-1} + O\left(|z|^{2q+1}\right) ,
$$

\subsection{Centrally symmetric bifurcations in the reversible conservative case}\label{sec:SymmCase_1}

We consider the flow of the form
\begin{equation}
\dot z = i\mu z + \sum\limits_{m,k\geq 0, \\ m+k\geq 2} \alpha_{m,k} z^m (z^*)^k ,
\label{eq:2fl20}
\end{equation}

\begin{lm}
If system (\ref{eq:2fl20}) is reversible with respect to the involution $z\to z^*$, $z^*\to z$ and time reversal $t\to -t$, then all the coefficients $\alpha_{m,k}$ are purely imaginary.
\label{lm:revfl}
\end{lm}

\begin{proof} The reversibility of~(\ref{eq:2fl20}) means that it remains invariant under the change $t\to -t$ and $z\to z^*$, $z^*\to z$.
Indeed, after this change, system~(\ref{eq:2fl20}) is rewritten as $- \dot z^* = i\mu z^* + \sum \alpha_{m,k} (z^*)^m (z)^k$. The corresponding system for $\dot z$ is 
$\dot z = i\mu z - \sum \alpha_{m,k}^* z^m (z^*)^k$ that coincides with~(\ref{eq:2fl20}) if $\alpha_{m,k} = - \alpha_{m,k}^*$.
\end{proof} 

Assume that~\eqref{eq:2map}  is conservative and reversible. Then the flow normal form takes the form~\eqref{eq401} with all the coefficients being purely imaginary.  

Let us consider the following truncated $2q$-order normal form of~(\ref{eq401}), 
\begin{equation}\label{eq425}
\dot z =
i\mu z + i\Phi(|z|^2)z + i \alpha (z^*)^{2q-1}  ,
\end{equation}
This system is Hamiltonian and  possesses certain symmetries.

\begin{itemize}
	\item[S1.] Flow~\eqref{eq425} is symmetric with respect to the rotation by  angle $\frac{\pi}{q}$, i.e. that it is invariant under the change
$$
\displaystyle z \to z e^{i\frac{\pi}{q}}
$$
Indeed, after this change, (\ref{eq425}) is rewritten as follows
$$
\begin{array}{l}
\dot z e^{i\frac{\pi}{q}} =
i\mu z e^{i\frac{\pi}{q}} + i\Phi(|z|^2)z e^{i\frac{\pi}{q}} + i\alpha (z^* e^{-i\frac{\pi}{q}})^{2q-1} = \\
\qquad\qquad\qquad\qquad = e^{i\frac{\pi}{q}} \left( i\mu z + i\Phi(|z|^2)z + i \alpha (z^*)^{2q-1}
e^{-i{2\pi}}\right) ,
\end{array}
$$
that shows that the system has not changed.

\item[S2.] System~(\ref{eq425}) is reversible with respect to the involution $z\to z^*$, $z^*\to z$ (in fact, $x\to x$, $y\to -y$) and time reversal $t\to -t$.
Therefore, the reversibility and the rotation symmetry
imply also that, in the $(x,y)$ plane, there exists the set of straight lines with slope $\psi = k \frac{\pi}{q}, \;k= 0,1...,q$, which are the lines of fixed points of the corresponding involution
$$
z e^{ik\frac{\pi}{q}} \to z^* e^{-ik\frac{\pi}{q}}
$$
and time reversal $t\to -t$.

\item[S3.] 
Under rotation by angle $\frac{\pi}{2q}$, i.e. after the change $z\to z e^{i \frac{\pi}{2q}}$, $z^*\to z^* e^{-i \frac{\pi}{2q}}$, 
(\ref{eq425}) takes the form
\begin{equation}\label{eq425n}
\dot z =
i\mu z + i\Phi(|z|^2)z - i \alpha (z^*)^{2q-1}  ,
\end{equation}
that is anti-symmetric to~(\ref{eq425}), since it differs only in the sign before the term $(z^*)^{2q-1}$. Thus, system~(\ref{eq425n}) is also reversible and has the same symmetries as~(\ref{eq425}).
\end{itemize}

Therefore, finally, we obtain the following result

\begin{lm}
In the $(x,y)$-plane there exists the set of straight lines $\psi = k \frac{\pi}{2q}, \;k= 0,1...,2q$, which are the lines of fixed points of the corresponding involution
$$
h_k: \;\;z e^{ik\frac{\pi}{2q}} \to z^* e^{-ik\frac{\pi}{2q}}
$$
for system~(\ref{eq425}).
\label{lm:sym2q}
\end{lm}

Note that system (\ref{eq425}) is also Hamiltonian, since the complex conjugate has the form
$$
\dot z^* =
-i\mu z^* - i\Phi(|z|^2)z^* - i \alpha z^{2q-1}  ,
$$
and, thus,
$$
\mbox{Div}\; = \; \frac{\partial \dot z}{\partial z} + \frac{\partial \dot z^*}{\partial z^*} = \left(i\mu +i \Phi(|z|^2) +  i \Phi^\prime(|z|^2) zz^*\right) +
\left(-i\mu - i \Phi(|z|^2)- i \Phi^\prime(|z|^2) zz^*\right) \equiv 0
$$
In order to find the Hamiltonian function of~(\ref{eq425}), one can introduce the complex polar coordinates.
Namely, let $z= \sqrt{r} e^{i\psi}$, where $r=x^2+y^2$. Then we obtain from (\ref{eq425})
$$
\begin{array}{l}
\frac{1}{2\sqrt{r}}\dot r  + i \sqrt{r}  \dot\psi = \\
= i\mu \sqrt{r}  + i \Phi(r) \sqrt{r} + i \alpha \frac{r^q}{\sqrt{r}} e^{-i2q\psi} .
\end{array}
$$
Separating the real and imaginary parts of this equality, we obtain the following system
\begin{equation}\label{eq47}
\begin{array}{l}
\dot r =  2\alpha r^q \sin(2q\psi)   , \\
\dot\psi = \mu +
\Phi(r) + \alpha r^{q-1}\cos(2q\psi) .
\end{array}
\end{equation}
One can easily see that~\eqref{eq47} is Hamiltonian 
$$
\dot r  = - \frac{\partial H }{\partial\psi}, \dot\psi = \frac{\partial H}{\partial r},
$$
with the Hamiltonian function
$$H(r,\psi) = \mu r + \int{\Phi(r)dr} + \frac{\alpha}{q}r^q\cos(2q\psi).
$$

Now we can analyze the dynamics and bifurcations of~(\ref{eq425}).
Since $\Phi(r) = l_1 r (1+ O(r))$, where $l_1\neq 0$, for the coordinates $(r,\psi)$ of the equilibria of~(\ref{eq425}), we obtain the following relations
$$2\alpha r^q \sin(2q\psi)= 0,\;\; \mu + l_1r(1+ O(r)) + \alpha r^{q-1}\cos(2q\psi) =0.
$$
Thus, if $\mu l_1 >0$, there are no nonzero equilibria. If $\mu l_1<0$,
there appear $4q$ nonzero equilibria $(\psi_*,r_*)$ that satisfy the equations $\sin(2q\psi_*)= 0$ and $r_* = -\mu/l_1 + O(\mu^2)$. They are located uniformly in the circumference of radius $r_*$ with a difference in angle $\frac{\pi}{2q}$.

As follows from~(\ref{eq47}), the characteristic polynomial for these equilibria has the form
$$
\left(\begin{array}{cc} -\lambda & 4 \alpha q r_*^q\cos( 2q\psi_*) \\   \Phi'(r_*)+ \alpha(q-1) r_*^{q-2} \cos(2q\psi_*) & -\lambda  \end{array} \right) = 0
$$
and, thus, the eigenvalues of the equilibria satisfy 
$$
\lambda^2 = - 4\alpha q l^{1-q} \mu^{q}\cos(2q\psi_*)\left(1+ O(\mu)\right),
$$
where $\cos(2q\psi) = +1$ or $\cos(2q\psi) = -1$ for alternating $\psi_*$:
\begin{itemize}
\item
$\psi_* = \frac{\pi}{q}k, k = 0,1,...,2q-1$, where  $\cos(2q\psi_*) = +1$.  These equilibria are saddles if $\alpha\mu<0$ (or $\alpha l_1>0$) and centers if $\alpha\mu>0$ (or $\alpha l_1<0$);
\item
$\psi_* = \frac{\pi}{2q} + \frac{\pi}{q}k, k = 0,1,...,2q-1$, where $\cos(2q\psi_*) = -1$. These equilibria are centers if $\alpha \mu<0$ ($\alpha l_1>0$) and saddles if $\alpha \mu>0$ ($\alpha l_1<0$).
\end{itemize}

Thus, when $\mu$ varies, we observe the bifurcation related to the appearance, for $\mu l_1 <0$, of a garland consisted of $4q$ equilibria, alternating $2q$ saddles and $2q$ centers, see Fig.~\ref{bdiag1p3n2} for
the cases of the centrally symmetric bifurcations of the 1:3 and 1:5  resonances.

\subsection{Centrally symmetric bifurcations in the general conservative case}\label{sec:SymmCase_2}

Let us consider the general (including non-reversible) conservative case with the coefficient $\delta = \alpha e^{i\vartheta}$, where $0\leq \vartheta <2\pi$, before the term $(z^*)^{2q-1}$. It is known that $\mbox{Re} \delta =0$ if $\vartheta = \pi/2, 3\pi/2$. We assume that, in general, $\mbox{Re} \delta \neq 0$.
We introduce a new coordinate $z = z_{new} e^{i\omega}$. Then flow~(\ref{eq401}) takes the form
$$
\dot z  =
i\mu z  + i\Phi(|z|^2)z  + \alpha e^{i\vartheta} e^{-i2q\omega}  (z^*)^{2q-1}  + O\left(|z|^{2q+1}\right) ,
$$
Hence, we obtain that $\delta_{new} = \alpha e^{i\vartheta} e^{-i2q\omega} = \alpha  e^{i(\vartheta - 2q\omega)}$ and, thus, we can have a new coefficient $\delta_{new} =  \alpha i$ if 
\begin{equation}\label{eq4wn2}
\omega = (\vartheta - \frac{\pi}{2})/2q .
\end{equation}

Thus, the non-reversible conservative case (when $\mbox{Re} \delta \neq 0$) is reduced to the reversible conservative case by means of the rotation of $(x,y)$-coordinates by $\omega$ given in~\eqref{eq4wn2}.
Therefore, dynamics and bifurcations of system (\ref{eq401}) are similar (up to some rotation) in the reversible (with time reversal and the involution $z\to z^*$, $z^*\to z$) and non-reversible cases.
Thus, for system (\ref{eq401}), when $\mu$ varies, if $\mu l_1 <0$, we observe a centrally symmetric bifurcation related to the appearance of a garland $G_{2q,2q}$ consisted of $4q$ equilibria, alternating $2q$ saddles and $2q$ centers, and this garland is rotated by $\omega$ compared to the reversible case $\mbox{Re} \delta = 0$ illustrated in  Fig.~\ref{bdiag1p3n2}  for the 1:3 and 1:5 resonances.

\section{Symmetry-breaking conservative  bifurcations. Proof of Theorem~\ref{th:th2}}  \label{sec:pfork}

In this section we consider symmetry-breaking bifurcations in map~(\ref{eq:2map}) that possesses the central symmetry $z\to -z, z^*\to -z^*$. We introduce two parameters $\mu_1$ and $\mu_2$, where $\mu_1$ is a parameter varying eigenvalues of the $p$:$q$ resonant fixed point. This parameter does not destroy the central symmetry, see Section~\ref{sec:SymmCase}, where $\mu_1=\mu$.
The parameter $\mu_2$ is introduced in order to activate one of the ``frozen'' resonant terms that vanishes in the presence of the symmetry. Among the possible  monomials of even degrees considered in Section~\ref{sec:NFRes} that satisfy the resonant condition~\eqref{eq:rezcond}, term $A_{0,q-1}(z^*)^{q-1}$ has the minimal degree and, therefore, we take it assuming $\mu_2\sim A_{0,q-1}$. Namely, we   consider the following two-parameter flow family which is the truncated $2q$-order normal form of~\eqref{eq402}
\begin{equation}\label{eq51}
\dot z =
i\mu_1 z + i\Phi(|z|^2)z + i\mu_2 (z^*)^{q-1} + i \alpha (z^*)^{2q-1}, 
\end{equation}
where $\mu_1$, $\mu_2$ and $\alpha\neq 0$ are real.
Note that system~(\ref{eq51}) is conservative and reversible with respect to the involution $z\to z^*, z^*\to z$ and time reversal $t\to -t$. We can benefit of this property, since centrally symmetric orbits can undergo pitchfork bifurcations.

In the polar coordinates $(r,\psi)$, where $z = \sqrt{r} e^{i\psi}$, system~(\ref{eq51}) takes the form
\begin{equation}\label{eq53}
\begin{array}{l}
\dot r =   2\mu_2 r^{q/2} \sin (q\psi) + 2 \alpha r^{q}\sin(2q\psi), \\ 
\dot\psi = \mu_1 + \Phi(r) + \mu_2 r^{q/2 - 1} \cos(q\psi) +
\alpha r^{q-1}\cos(2q\psi), 
\end{array}
\end{equation}
This system is Hamiltonian with the Hamiltonian function
\begin{equation}\label{eq53H}
H(r,\psi) = \mu_1 r + \int{\Phi(r)dr} + \frac{2}{q}\mu_2 r^{q/2} \cos q\psi + \frac{\alpha}{q}r^q\cos 2q\psi.
\end{equation}

Nonzero equilibria of~(\ref{eq53}) satisfy the equations
\begin{equation}\label{eq532}
\begin{array}{l}
\sin (q\psi) \left(\mu_2  + 2\alpha r^{q/2}\cos (q\psi)\right) = 0, \\
\mu_1 + \Phi(r) + r^{q/2 -1}\left(\mu_2 \cos (q\psi) + \alpha r^{q/2}\cos (2q\psi) \right)=0.
\end{array}
\end{equation}
From these equations we can deduce that the nonzero equilibria are in the circumference centered at $z=0$ of radius
$$r_*=-\frac{\mu_1}{l_1}+O_{\min\{q/2,2\}} (\mu_1,\mu_2).$$
As for the $\psi$-coordinate of the equilibria, we note that in~\eqref{eq532}, in the first equation  the factor $\sin (q\psi) = 0$ defines the coordinates $\psi_s=\frac{\pi}{q}k, k=0,\pm 1, \ldots$, of centrally symmetric equilibria, 
 while
the factor
$$
\mu_2  + 2\alpha r^{q/2}\cos (q\psi) =0
$$
gives the coordinates of centrally non-symmetric equilibria $\psi_{ns}$. Evidently, these equilibria exist when $-1\leq\cos(q\psi_{ns})\leq 1$, i.e. for values of
the parameters $(\mu_1,\mu_2)$ inside the domain lying between the branches of curve $L_{pf}$ given by~\eqref{eq540} (domain $III$ of Theorem~\ref{th:th2}). 
This curve corresponds to a pitchfork bifurcation of centrally symmetric equilibria $\psi_s$: for $\alpha>0$,  the upper branch of the curve corresponds to the pitchfork bifurcation of equilibria $\psi_s = \pi/q (1 + 2 k)$; 
the lower branch of $L_{pf}$ corresponds to the pitchfork bifurcation of equilibria 
$\psi_s = 2\pi k/q$, and otherwise for $\alpha<0$.

Varying parameters $\mu_1$ and $\mu_2$ further inside domain $III$ (between the branches of $L_{pf}$) one can look at the changes that the Hamiltonian function~\eqref{eq53H} undergoes. One can observe that the separatrices of the saddles modify. For example, in the case $\alpha l_1>0$, the separatrices of $q$ centrally symmetric interior saddles move toward the separatrices of the other centrally symmetric exterior saddles. At some moments (in the bifurcation curves $g_1$ and $g_2$) they connect and splitting of separatrices arises. One can calculate the corresponding Melnikov integrals in order to approximate rigorously these moments (which is out of scope of this paper).

Note that for $\mu_2=0$, the centrally symmetric equilibria are saddles in the case $\mu_1>0$, $\alpha>0$ and centers in the case $\mu_1>0$, $\alpha<0$.

\section{On symmetry-breaking reversible non-conservative  bifurcations.}\label{sec:NonCons}

In this section, we discuss bifurcations of the symmetric $p$:$q$ resonant elliptic point in families $\bar z=f(z_1,z_2,\mu_1, \mu_2)$ that destroy both the symmetry and conservativity but keep reversibility. Note that this kind of perturbations can be obtained using, for instance, the so-called Quispel-Roberts~\cite{RQ92}, cross-form~\cite{GGS21} and antisymmetric~\cite{GSZ22} perturbations.     
We follow the ideas by~\cite{GLRT14,GT17} to study symmetry-breaking bifurcations for reversible perturbations of conservative systems.
We consider a two-parameter family slightly different of system~(\ref{eq402}): 
\begin{equation}\label{eq403}
\begin{array}{l}
\dot z =
i\mu_1 z + i \Phi(|z|^2)z + i\mu_2 (z^*)^{q-1} + i A \mu_2 z^{q+1} + i B \mu_2 z (z^*)^q  + i\alpha (z^*)^{2q-1} + \\
\qquad\qquad\qquad\qquad + O\left(|z|^{2q+1}\right) + O\left(|\mu||z|^{q+2}\right),
\end{array}
\end{equation}
where $A$ and $B$ are real coefficients as well as $\mu_1$, $\mu_2$ are parameters, $\Phi$ takes the form~\eqref{eq:Phi} and $\alpha$ is a real coefficient as in previous sections. This normal form includes new  resonant terms that break down the conservativity. In Section~\ref{sec:NFRes}, such terms that satisfy the resonant condition~\eqref{eq:rez0} and have minimal degree are  
$A_{q+1,0} z^{q+1}$ and $A_{1,q}z(z^*)^q$.
Then in flow~\eqref{eq403}, 
the new coefficients $iA\mu_2$ and $iB\mu_2$ correspond to $A_{q+1,0}$ and $A_{1,q}$, respectively, and they are purely imaginary in order to  preserve the reversibility under the involution $z\to z^*$, $z^*\to z$ and time reversal $t\to -t$, see Lemma~\ref{lm:revfl}.

Flow~\eqref{eq403} is no longer conservative, since the divergence
\begin{equation}\label{eq:DivRevNoCons} 
\mbox{Div} \; \equiv \; \frac{\partial \dot z}{\partial z} + \frac{\partial \dot z^*}{\partial z^*} =
i\mu_2((q+1)A -B)(z^q - (z^*)^q) + O(|z|^{2q}) 
\end{equation}
is not identical zero if $B\neq A(q+1)$. The divergence also implies that
the reversibly symmetric equilibria (for which $z^q = (z^*)^q$) should be conservative. 

\begin{figure}[t]
	\centerline{\includegraphics[width=12cm]{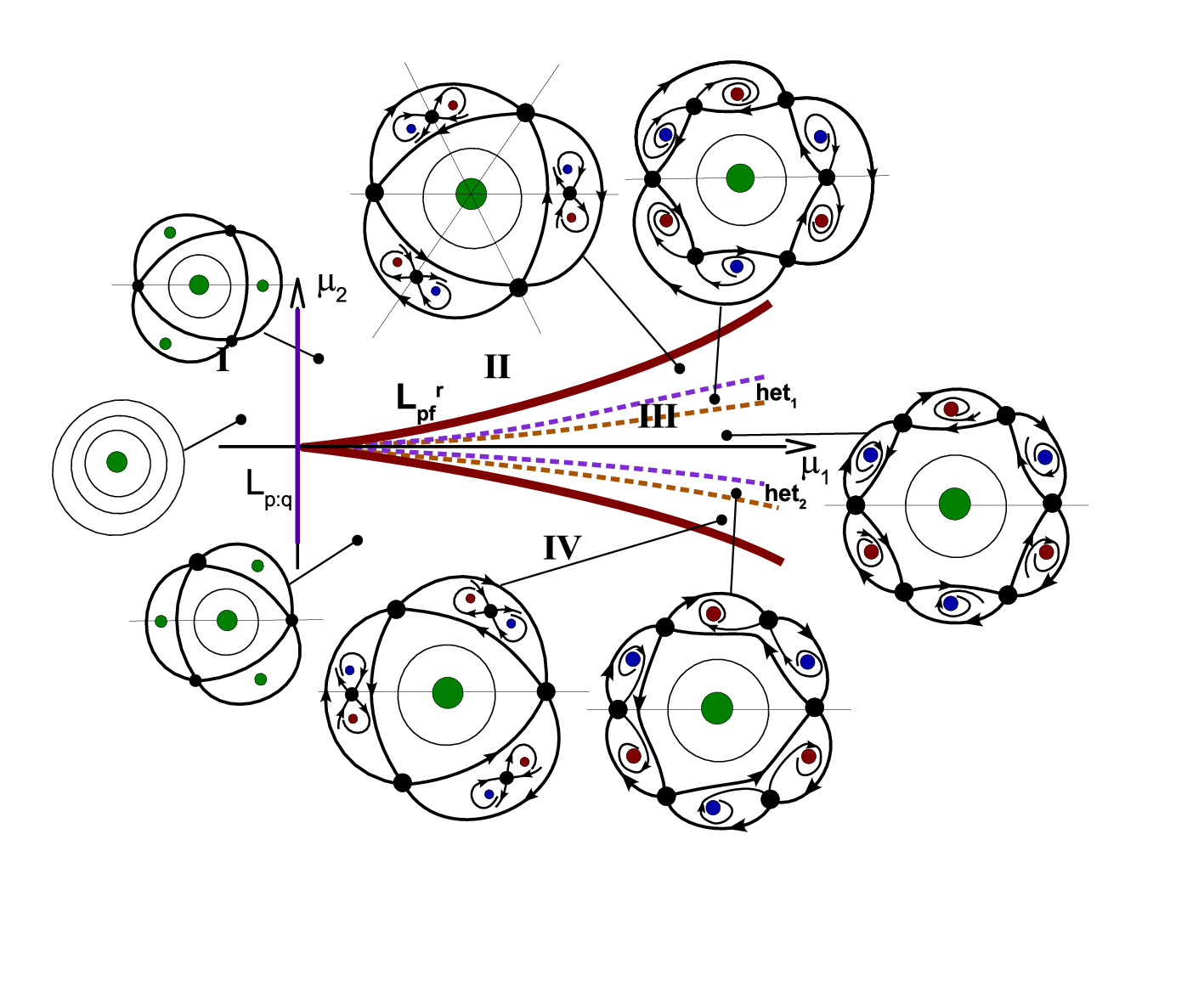}}
	\caption{{\footnotesize A reversible bifurcation of $p$:$q$ resonance with odd $q\geq 3$ for a two-parameter family where both the central symmetry and conservativity are broken. Main elements of the bifurcation diagram for the case of the 1:3 resonance with $l_1\alpha >0$ are illustrated. The bifurcation diagram in the case $l_1\alpha<0$ is similar to Fig.~\ref{bifpdf1p4q}b for the conservative flow~(\ref{eq402}) 
	}}
	\label{bdrev1p3v2}
\end{figure}

However, reversibly non-symmetric orbits can be non-conservative, and they appear due to pitchfork bifurcations  as in Section~\ref{sec:pfork}. In the $(\mu_1,\mu_2)$-parameter plane, the equation of the corresponding reversible bifurcation curve $L^r_{pf}$ is asymptotically  close (up to the $O$-terms) to curve $L_{pf}$ given in~(\ref{eq540}). However, the corresponding bifurcation is different, it is a reversible pitchfork bifurcation of reversibly symmetric equilibria~\cite{LT12}. If the equilibrium is a center, it becomes a reversibly symmetric saddle and a reversibly symmetric pair of new equilibria emerges such that one equilibrium is a stable focus the other is an unstable focus (supercritical reversible pitchfork bifurcation). If the equilibrium is a saddle, then it becomes a reversibly symmetric center and a reversibly symmetric pair of saddle equilibria with positive and negative divergence appears (subcritical reversible pitchfork bifurcation). As follows from~(\ref{eq:DivRevNoCons}), if $B\neq A(q+1)$, these non-symmetric equilibria have the divergence of opposite sings. Thus, they are either non-conservative saddles (with $\mbox{Div}_1 = - \mbox{Div}_2 \neq 0$) or stable and unstable foci.
Thus, Theorem~\ref{th:th2} can be g eneralized to the case of small reversible perturbations that do not preserve the conservativity. Note that similar result was proved in \cite{GT17} for $q$-order normal form.

Let $V_0$ be a small neighbourhood of the origin in the $(\mu_1,\mu_2)$-parameter plane and 
let $L^r_{pf}$ and $L_{p:q}$ be the bifurcation curves   that
correspond  to a reversible pitchfork bifurcation of $q$ reversibly symmetric nonzero equilibria and the $p$:$q$ resonance, respectively.
Then  curves $L_{pf}^r$ and $L_{p:q}$ divide $V_0$ into 4 domains $I$-$IV$ with the following dynamics of~(\ref{eq403}), see Fig~\ref{bdrev1p3v2}. 
Flow~\eqref{eq403} has no nonzero equilibria in domain $I$.
In domains $II$ and $IV$, there appears a garland $G_{q,q}$ which consists of $q$ reversibly symmetric conservative saddles and $q$ reversibly symmetric conservative centers. 
In domain $III$, the dynamics of~(\ref{eq403}) is reversible and non-conservative if $B\neq A(q+1)$. Moreover, 
if $l_1\alpha <0$, after reversibly symmetric saddles undergo subcritical reversible pitchfork bifurcations $L^r_{pf}$, system~(\ref{eq403}) has $4q$ nonzero equilibria which form a non-conservative garland $G'_{2q,q,q}$ which consists of $2q$ conservative and reversibly symmetric centers and $q$ reversibly symmetric pairs of saddles with positive and negative divergence. However, we need to stress that~(\ref{eq403}) is not Hamiltonian and it has an integrating factor that is singular at the saddle equilibria. The phase portraits of~(\ref{eq403}) for $l_1\alpha<0$ are similar to the ones in Fig.~\ref{bifpdf1p4q}b for flow~(\ref{eq402}).

As for $l_1\alpha >0$, after a supercritical reversible pitchfork bifurcation of reversibly symmetric centers, system~(\ref{eq403}) has $4q$ nonzero equilibria in domain $III$. The $2q$ equilibria are reversibly symmetric saddles and the other $2q$ equilibria compose $q$ reversibly symmetric pairs of stable and unstable foci. All these equilibria compose a a new type of a non-conservative garland $G'_{2q,q,q}$ consisting of $2q$ conservative reversibly symmetric saddles, $q$ stable foci and $q$ unstable  foci. At varying parameters inside domain $III$, garland $G'_{2q,q,q}$ should pass through various heteroclinic reconstructions in bifurcation curves that correspond to heteroclinic connections between the saddles.

For the map, these garlands corresponds to discrete garlands $G'_{2q,q,q}$ with two reversibly symmetric $q$-periodic saddle orbit and $q$-periodic stable focus and $q$-periodic unstable focus.

\section{Acknowledgments}

The author thanks S.~V.~Gonchenko and D.~V.~Turaev for fruitful discussions. The author also acknowledges the Serra H\'unter program, the Spanish grant PID2021-125535NB-I00 (MICINN/AEI/FEDER, UE) and the Catalan grant 2021-SGR-01072.

\end{document}